\documentclass[12pt]{amsart}
\usepackage{amsmath,mathrsfs,cite}

\usepackage[ps2pdf,colorlinks=true,urlcolor=blue,
citecolor=red,linkcolor=blue,linktocpage,pdfpagelabels,bookmarksnumbered,bookmarksopen]{hyperref}
\usepackage[english]{babel}

\usepackage[left=2.7cm,right=2.7cm,top=2.94cm,bottom=2.94cm]{geometry}

\numberwithin{equation}{section}

\newcommand{\R}{{\mathbb R}}

\newcommand{\eps}{\varepsilon}

\newcommand{\intr}{{\int_{\R^N}}}

\newcommand{\refe}[1]{{(\ref{#1})}}

\renewcommand{\theta}{\vartheta}

\newcommand{\wsobr}{W^{1,p}_{{\rm rad}}(\R^{N})}

\numberwithin{equation}{section}
\newtheorem{theorem}{Theorem}[section]
\newtheorem{proposition}[theorem]{Proposition}
\newtheorem{lemma}[theorem]{Lemma}
\newtheorem{remark}[theorem]{Remark}

\newtheorem{definition}[theorem]{Definition}
\theoremstyle{definition}

\newcommand{\beq}{\begin{equation}}
\newcommand{\eeq}{\end{equation}}

\title[mountain pass solutions for quasi-linear equations]{Mountain pass 
solutions for quasi-linear \\ equations via a monotonicity trick}

\author{Benedetta Pellacci}
\address{Dipartimento di Scienze Applicate
\newline\indent
Universit\`a di Napoli Parthenope
\newline\indent
Isola C4, I-80143 Napoli, Italy}
\email{pellacci@uniparthenope.it}

\author{Marco Squassina}
\address{Dipartimento di Informatica
\newline\indent
Universit\`a degli Studi di Verona
\newline\indent
C\'a Vignal 2, Strada Le Grazie 15, I-37134 Verona, Italy}
\email{marco.squassina@univr.it}

\thanks{The second author was partially supported 
by 2007 MIUR Project: {\em Metodi Variazionali e Topologici
nello Studio di Fenomeni non Lineari}}

\begin{document}
	



\begin{abstract}
We obtain the existence of symmetric mountain pass solutions for  quasi-linear 
equations without the typical assumptions which guarantee the 
boundedness of an arbitrary Palais-Smale sequence. This is  done through a recent 
version of the monotonicity trick proved in \cite{montrick}. The main results
are new also for the $p$-Laplacian operator.
\end{abstract}
\maketitle

\section{Introduction}
Let $N>p>1$. In the study of the quasi-linear partial differential equation
\begin{equation}
	\label{eq-iintro}
-{\rm div}(j_{\xi}(u,Du))+j_{s}(u,Du)+V(x) |u|^{p-2}u=g(u),\quad u\in W^{1,p}(\R^N)
\end{equation}
by means of variational methods, a rather typical assumption on $j(s,\xi)$ and $g(s)$ is that
there exist $p<q<Np/(N-p)$, $\delta>0$ and $R\geq 0$ such that
\begin{equation}
	\label{generalARcond}
q j(s,\xi)-j_s(s,\xi)s-(1+\delta)j_\xi(s,\xi)\cdot\xi- q G(s)+g(s)s\geq 0,
\end{equation}
for all $s\in\R$ such that $|s|\geq R$ and any $\xi\in\R^N$ (cf.~\cite{AB,candeg}). This condition ensures that {\em every} 
Palais-Smale sequence, in a suitable sense, of the associated functional $f:W^{1,p}(\R^N)\to\R,$
\begin{equation*}
f(u)=\intr j(u,Du)+\frac{1}{p}\intr V(x)|u|^{p}-\intr G(u),
\end{equation*}
is {\em bounded} in $W^{1,p}(\R^N)$. We might refer to this technical
condition as the generalized Ambrosetti-Rabinowitz condition, involving
the terms of the quasi-linear operator $j$. In fact, in the treatment of the non-autonomous semi-linear equation
\begin{equation}
	\label{semieq-iintro}
-\Delta u+V(x) u=g(u),\quad u\in H^1(\R^N),
\end{equation}
the previous inequality \eqref{generalARcond} reduces to the classical
Ambrosetti-Rabinowitz condition \cite{AR}, namely $0<q G(s)\leq g(s)s$, 
for every $s\in\R$ with $|s|\geq R$. Of course, aiming to achieve the existence of
{\em multiple} solutions for equation \eqref{eq-iintro}, one needs to know that the Palais-Smale condition
for $f$ is satisfied at an {\em arbitrary} energy level, and hence it is necessary to guarantee 
that Palais-Smale sequences are always at least bounded, through condition \eqref{generalARcond}.
On the contrary, under suitable assumptions, if one merely focuses on the 
existence of a nonnegative Mountain Pass solution of \eqref{eq-iintro}, it is reasonable
to expect that by a clever selection of a {\em special} Palais-Smale sequence at the Mountain Pass level $c$
one could reach the goal of getting a solution to \eqref{eq-iintro} without knowing that the Palais-Smale condition holds. 
The existence of such a nice sequence is possible since the definition
of $c$ allows to detect continuous paths $\gamma:[0,1]\to W^{1,p}(\R^N)$ with a very good behavior.
The idea, considering for instance problems \eqref{semieq-iintro}, is to see $f=f_1$ as the end point of the continuous 
family of $C^1$ functionals $f_\lambda:H^1(\R^N)\to\R,$
$$
f_{\lambda}(u)=\frac12\intr |Du|^2+\frac12\intr V(x)|u|^{2}-\lambda\intr G(u).
$$
When $f_{\lambda}$ satisfies a uniform Mountain Pass geometry, then it is possible to use  the 
so called monotonicity trick for $C^1$ smooth functionals, originally discovered by Struwe \cite{struwem} in a very special setting 
and generalized and formalized later in an abstract framework
by Jeanjean \cite{Jj} and Jeanjean-Toland \cite{JT}. This strategy 
provides a bounded Palais-Smale sequence for all $\lambda$ fixed, up to a set of null measure.
Then, by requiring some compactness condition one can detect 
a sequence  $(\lambda_j)$, increasingly converging to $1$, for which there corresponds a 
sequence $(u_{\lambda_j})$ of solutions to \eqref{semieq-iintro} at the Mountain Pass level $c_{\lambda_j}$, namely
\begin{equation}
	\label{minmaxclass}
c_\lambda=\inf_{\gamma\in \Gamma}\sup_{t\in [0,1]} f_\lambda(\gamma(t)),\quad\,\,
\Gamma=\{\gamma\in C([0,1],W^{1,p}(\R^N)):\gamma(0)=0,\gamma(1)=w\},
\end{equation}
being $w\in W^{1,p}(\R^N)$ a suitable function with $f_\lambda(w)<0$ for any value of $\lambda$.
Then, being $u_{\lambda_j}$ exact solutions, one can exploit the Poh\v ozaev identity and combine it with the energy level constraint 
to show in turn that $(u_{\lambda_j})$ is a {\em bounded} Palais-Smale sequence for $f_1$.
In the case of {\em semi-linear} equations such as \eqref{semieq-iintro}, we refer the reader to 
\cite{Jtanak,azzpomp} where the approach has been successfully developed. 
\vskip1pt
The main goal of this manuscript is twofold. On one hand, we intend 
 to show how condition \eqref{generalARcond} can be completely removed
by using a general version of the monotonicity 
trick recently developed in \cite{montrick} in the framework of 
the non-smooth critical point theory of \cite{dm,cdm}. 
In this respect, first, in order to analyze the most clarifying concrete 
situation, we consider a class of functionals invariant under 
orthogonal transformations, set in the space of radial functions (see Theorem~\ref{rad}). 
As in the smooth case, by studying a penalized functional
$f_{\lambda}$ we will obtain a sequence of $\lambda_{j}$ converging 
to one, with corresponding weak solutions $u_{\lambda_{j}}$. 
In order to obtain that the sequence  $(u_{\lambda_{j}})$ is bounded, a general version of the 
Poh\v ozaev identity \cite{dms} for merely $C^1$ weak solutions
will be crucial, as $C^{1,\alpha}$ is the optimal regularity if $p\neq 2$ \cite{tolks}. 
Moreover, a generalized version of the Palais' symmetric criticality principle recently achieved in \cite{sqradII} will be exploited.
These results are new also in the particular meaningful case $j(u,Du)=|Du|^{p}/p$ with $p\neq 2$,
being the case $p=2$ covered in \cite{azzpomp}.
On the other hand, when one does not restrict the functional to the space of 
radially symmetric functions (see Theorem~\ref{sym}), it is possible to make a stronger use of the 
result in \cite{montrick} to construct a bounded, almost symmetric (cf.\ \eqref{simmcondtiondef}), Palais-Smale sequence 
which will give a radial and radially decreasing solution. At the high level 
of generality of equation~\eqref{eq-iintro}, proving a priori that the radial solution 
is decreasing seems a particularly strong fact. These results are new also 
for $j(u,Du)=|Du|^{p}/p$, even with $p=2$.
\vskip2pt

Let us now state the main results of the paper. 
Let $N> p>1$ and let $j:\R\times \R^{+}\to \R^{+}$ be a $C^{1}$ function 
such that the map $t\mapsto j(s,t) $ is increasing and strictly convex. 
Moreover, we assume that there exist $\alpha,\beta >0$ with
\begin{gather}
	\label{j0}
\alpha t^{p}\leq j(s,t)\leq \beta t^{p}, \qquad\text{for every $s\in \R$ and  $t\in \R^{+}$,}\\
\label{jder}
|j_{s}(s,t)|\leq \beta t^{p},\quad\,\, | j_{t}(s,t)|\leq \beta t^{p-1}, \qquad\text{for every $s\in \R$ and  $t\in \R^{+}$,}\\
\label{segno}
j_{s}(s,t)s\geq 0, \qquad\text{for every $s\in \R$ and  $t\in \R^{+}$}.
\end{gather}
Let $V:\R^+\to\R^+$ be a $C^1$ function such that there exist $m,M\in \R^+$ with
\beq\label{V0}
0<m\leq V(\tau)\leq M, \qquad\text{for every $\tau\in \R^{+}$.}
\eeq
Furthermore, we shall assume that
\beq\label{V1}
\|V'(|x|)|x|\|_{L^{N/p}(\R^N)}<\alpha p{\mathcal S},
\eeq
where $\alpha$ is the number appearing in \eqref{j0}
and ${\mathcal S}$ is the best Sobolev constant.
Apart from the natural growths \eqref{j0}-\eqref{jder}, conditions \eqref{segno} is a 
typical requirement in the frame of quasi-linear 
equations, which helps \cite{AB,candeg,pelUMI,squTO,monog}
in the achievement of both existence and summability
issues related to equation \eqref{eq-iintro}.
Under \eqref{j0} and \eqref{V0}, the functional
defined either in $W^{1,p}_{{\rm rad}}(\R^N)$  or in  
$W^{1,p}(\R^N)$ as 
$$
u\mapsto \intr j(u,|Du|)+V(|x|)\frac{|u|^p}{p},
$$
is continuous but not even locally Lipschitz, as it can be easily checked. Moreover, it
admits Gateaux derivatives along any bounded direction $v$, but not on an arbitrary direction $v$ of 
either $W^{1,p}_{{\rm rad}}(\R^N)$ or $W^{1,p}(\R^N)$. This is the reason why we will make use of the abstract machinery developed
in \cite{dm,cdm} for continuous functionals, the related monotonicity 
trick proved in \cite{montrick} and  the Palais' symmetric criticality principle formulated in \cite{sqradII}.
\vskip4pt
\noindent
Let $p^*:=Np/(N-p)$ and consider the equation
\beq\label{equarad}
-{\rm div}\Big[j_{t}(u,|Du|)\frac{Du}{|Du|}\Big]+j_{s}(u,|Du|)+V(|x|) u^{p-1}=g(u)\quad\,\,\, \text{in $\R^{N}$}.
\eeq
Our first main result is the following

\begin{theorem}\label{rad}
Assume \eqref{j0}-\eqref{V1} and let $g:\R^{+}\to \R^{+}$ be 
continuous with $g(0)=0$ and extended by zero on $\R^-$. Moreover,
\beq\label{ipog}
\lim_{s\to 0^{+}}\frac{g(s)}{s^{p-1}}=\lim_{s\to+\infty}
\frac{g(s)}{s^{p^{*}-1}}=0,
\eeq
and, furthermore, for $G(s)=\int_0^s g(t)$,
\beq\label{Gpos}
\text{there exists $s>0$ such that $pG(s)-Ms^p>0$}.
\eeq
Then equation \eqref{equarad} admits a nontrivial, nonnegative, 
distributional and radially symmetric solution $u\in W^{1,p}(\R^{N})$.
\end{theorem}

\noindent
This result seems new even in the particular $p$-Laplacian case $j(s,t)=t^p/p$ with $p\neq 2$.
In order to prove Theorem~\ref{rad}, we consider the continuous functionals
$f_{\lambda}: W^{1,p}_{{\rm rad}}(\R^{N})\to \R$,
\begin{equation}
	\label{frad}
f_{\lambda}(u)=\intr j(u,|Du|)+\intr V(|x|)\frac{|u|^{p}}{p}-
\lambda\intr G(u),\qquad \lambda\in [\delta,1],
\end{equation}
for some suitable value of $\delta\in (0,1)$. 
First we shall prove that $f_\lambda$ fulfills a uniform Mountain Pass geometry.
Next we show that for all $\lambda\in (\delta,1]$ any bounded 
Palais-Smale sequence is, actually, strongly convergent. Furthermore, by applying the monotonicity trick
of \cite{montrick} and the Palais' symmetric criticality principle 
proved in \cite{sqradII} for continuous functionals, a sequence $\lambda_h\subset [\delta,1)$ with $\lambda_h\nearrow 1$ is detected
such that, for each $h\geq 1$, there exists a distributional solution $u_{\lambda_h}\in W^{1,p}_{{\rm rad}}(\R^{N})$ of
$$
-{\rm div}\Big[ j_{t}(u,|Du|)\frac{Du}{|Du|}\Big]+j_{s}(u,|Du|)+V(|x|) u^{p-1}=\lambda_h g(u)\qquad \text{in $\R^{N}$}
$$
at the Mountain Pass level $c_{\lambda_h}$.
Then, by exploiting a Poh\v ozaev identity \cite{dms} for $C^1$ solutions of \eqref{equarad}, 
we show in turn that $(u_{\lambda_h})$ is also a bounded Palais-Smale condition for $f_1$, 
and passing to the limit will provide the desired conclusion.
\vskip2pt
\noindent
Our second main result is the following

\begin{theorem}\label{sym}
Assume \eqref{j0}-\eqref{V1}, let $g:\R^{+}\to \R^{+}$ be continuous with $g(0)=0$,
extended by zero on $\R^-$, satisfying \eqref{Gpos}, and such 
that for all $\eps>0$ there is $C_{\eps}\in \R^{+}$ with
\beq\label{ipog2}
|g(s)|\leq \eps s^{p-1}+C_{\eps}s^{q-1},\quad p<q<p^{*},
\eeq
for every $s\in \R^{+}$. Let $V$ also satisfy 
\beq\label{pol}
|x|\leq |y| \,\,\;\Longrightarrow\,\,\; V(|x|)\leq V(|y|)\quad\text{for every  $x,y\in \R^{N}$.}
\eeq
Then equation \eqref{equarad} admits a nontrivial, nonnegative, distributional, 
radially symmetric and decreasing solution $u\in W^{1,p}(\R^{N})$.
\end{theorem}

\noindent
This result seems new even in the particular $p$-Laplacian case $j(s,t)=t^p/p$, included $p=2$.
In place of \eqref{ipog}, here we need the slightly more restrictive condition \eqref{ipog2},
since we cannot work directly on sequences of radial functions, which enjoy uniform decay properties.
In order to prove Theorem~\ref{sym}, we argue on the continuous  functionals  
$f_{\lambda}: W^{1,p}(\R^{N})\to \R$ again defined as in \eqref{frad}
for all $\lambda\in (\delta,1]$, for a suitable $\delta\in (0,1)$. 
Hence here we do not
restrict the functional to the space of radially symmetric functions. 
However, we still proceed as indicated above for the proof
Theorem~\ref{rad}, but, by exploiting the symmetry properties of the functional under polarization (cf.\ \cite{montrick}) we use 
the symmetry features of the monotonicity trick of \cite{montrick} and
we obtain the existence of a bounded and almost symmetric (cf.\ \eqref{simmcondtiondef}) Palais-Smale sequence for $f_1$.
Possessing a compactness result for such sequences, we can conclude the proof.
We remark that in this second statement the solution found is not only
radially symmetric, but also automatically {\em radially decreasing}. 
While in Theorem~\ref{rad} the solution is found at the {\em restricted} Mountain Pass level
$$
c_{{\rm rad}}=\inf_{\gamma\in \Gamma_{{\rm rad}}}\sup_{t\in [0,1]} f_1(\gamma(t)),\quad
\Gamma_{{\rm rad}}=\{\gamma\in C([0,1],W^{1,p}_{{\rm rad}}(\R^N)):\gamma(0)=0,\gamma(1)=w\},
$$
in Theorem~\ref{sym} the solution is found at the {\em global} Mountain Pass level 
$$
c=\inf_{\gamma\in \Gamma}\sup_{t\in [0,1]} f_1(\gamma(t)),\quad
\Gamma=\{\gamma\in C([0,1],W^{1,p}(\R^N)):\gamma(0)=0,\gamma(1)=w\}.
$$
Of course, on one hand, we have $c\leq  c_{{\rm rad}}$. On the other hand it is not clear if, in 
general, one has $c=c_{{\rm rad}}$ or $c<c_{{\rm rad}}$ although, precisely as a {\em further consequence} of 
Theorem \ref{sym}, this occurs when $V$ is constant and the map $t\mapsto j(s,t)$ is $p$-homogeneous
(see Remark~\ref{finalremark}).

\section{Proof of Theorem \ref{rad}}
\label{proofradthm}
\noindent
We will prove Theorem \ref{rad} by studying the functionals 
$f_{\lambda}: W^{1,p}_{{\rm rad}}(\R^{N})\to \R$ defined in \eqref{frad}. Taking into account
assumptions \eqref{j0}, \eqref{V0} and~\eqref{ipog}, recalling \cite[Theorem A.VI]{berlions},
it follows that $f_\lambda$ is well defined and (merely) continuous.
In turn, we shall exploit the non-smooth critical point theory of \cite{dm,cdm}
including the connection between critical points in a suitable sense and 
solutions of the associated Euler's equation (see for instance \cite[Theorem 3]{pelUMI} 
and also \cite[Proposition 6.16]{sqradII} for the symmetric setting). More precisely
under assumption \eqref{j0}-\eqref{V1}, the critical points of $f_{\lambda}$ are distributional solutions of 
\begin{equation}
	\label{formulaz-eq-lambda}
-{\rm div}\Big[ j_{t}(u,|Du|)\frac{Du}{|Du|}\Big]+j_{s}(u,|Du|)+V(|x|) |u|^{p-2}u=\lambda g(u)\qquad \text{in $\R^{N}$}.
\end{equation}
\vskip4pt
\noindent
Combining the following two lemmas shows that the minimax class \eqref{minmaxclass} is nonempty and
that the family $(f_\lambda)$ enjoys a uniform Mountain Pass geometry
whenever $\lambda$ varies inside the interval $[\delta_0,1]$, for a suitable $\delta_0>0$.

\begin{lemma}\label{geo1}
Assume \eqref{j0}, \eqref{V0} and \eqref{ipog}-\eqref{Gpos}. Then there exists 
$\delta_{0}\in (0,1)$ and a curve $\gamma\in C([0,1], W^{1,p}_{{\rm rad}}(\R^{N}))$, 
independent of $\lambda$, such that $f_{\lambda}(\gamma(1))<0$, for every $\lambda\in [\delta_{0},1]$. 
\end{lemma}
\begin{proof}
Due to \eqref{Gpos}, there exists 
$z\in W^{1,p}_{{\rm rad}}(\R^{N})$, $z\geq 0$ and Schwartz symmetric, such that
$$
\intr \Big(G(z)-\frac{M}{p} z^{p}\Big)>0.
$$
To see this, follow closely the first part of \cite[Step 1, pp.324-325]{berlions}. 
In turn, let $\delta_{0}\in (0,1)$ with
\begin{equation}
	\label{perturbposit}
\intr \Big(\delta_{0}G(z)-\frac{M}p z^{p}\Big)>0,
\end{equation}
and define the curve $\eta\in C([0,\infty), W^{1,p}_{{\rm rad}}(\R^{N}))$ by setting
$\eta(t):=z(\cdot/t)$ for $t\in (0,\infty)$ and $\eta(0):=0$. From
\eqref{j0} and \eqref{V0} it follows that
$$
f_{\lambda}(\eta(t))\leq \beta t^{N-p}\|Dz\|_{L^p(\R^N)}^{p}-t^{N}\intr\Big(\delta_{0}G(z)-\frac{M}p z^{p}\Big),
$$
yielding, on account of \eqref{perturbposit}, a time $t_0>0$ such that $f_{\lambda}(\eta(t_0))<0$ for every $\lambda\in[\delta_{0},1]$.
Then, the curve $\gamma\in C([0,1], W^{1,p}_{{\rm rad}}(\R^{N}))$, independent of $\lambda$, 
defined by $\gamma(t):=\eta(t_0t)$ has the required property and $\Gamma$ is nonempty by taking $w:=\gamma(1)$.
\end{proof}

\begin{lemma}\label{geo2}
Assume \eqref{j0}, \eqref{V0} and \eqref{ipog}. Let $\delta_0>0$ be the number found in Lemma~\ref{geo1}.
There exist $\sigma>0$ and $\rho>0$, independent of $\lambda$, such that $f_{\lambda}(u)\geq \sigma$
for any $u$ in $W^{1,p}_{{\rm rad}}(\R^{N})$ with $\|u\|_{1,p}=\rho$ and for every $\lambda\in[\delta_{0},1]$.
\end{lemma}
\begin{proof}
Condition \eqref{ipog} implies that for every $\eps>0$, there exists $C_{\eps}$ such that
\beq\label{crescita}
|g(s)|\leq \eps s^{p-1}+C_{\eps} s^{p^{*}-1},\qquad\text{for every $s\in \R^+$}.
\eeq
Then, fixed $\eps_{0}<m$, we find $C_{\eps_{0}}$ such that for every
$\lambda\in [\delta_{0},1]$ 
$$
f_{\lambda}(u)\geq \alpha \|Du\|_{L^p(\R^N)}^{p}+\frac{m-\eps_{0}}{p}\|u\|_{L^p(\R^N)}^{p}-C_{\eps_{0}}\|u\|_{W^{1,p}(\R^N)}^{p^{*}}.
$$
This last inequality immediately gives the conclusion.
\end{proof}

We will use the following compactness condition.

\begin{definition}
Let $\lambda,c\in \R$. We say that $f_{\lambda}$ 
satisfies the concrete-$(BPS)_{c}$ condition if any bounded sequence 
$(u_{h})\subset W^{1,p}_{{\rm rad}}(\R^{N})$ such that there is
$w_{h}\in W^{-1,p'}_{{\rm rad}}(\R^{N})$ with
\beq\label{bps}
f_{\lambda}(u_{h})\to c,\quad \langle f'_{\lambda}(u_{h}),v \rangle 
=\langle w_{h},v\rangle\quad\text{for every $v\in C^{\infty}_{c, {\rm rad}}(\R^{N})$}, \quad \text{and $w_{h}\to 0$} 
\eeq
admits a strongly convergent subsequence.
\end{definition}

\noindent
In the next result we will use the property 
\begin{equation}\label{jxicoerc}
j_t(s,t) t \geq \alpha t^p,
\end{equation}
which can be obtained by hypotheses $\refe{j0}$ 
once one has observed that, as $j$ is a strict convex function with respect to $t$, it results $0=j(s,0)\geq j(s,t)+j_t(s,t)\cdot(0-t).$

\begin{proposition}\label{comp}
Let $\lambda\in [\delta_{0},1]$, $c\in \R$ and assume 
\eqref{j0}-\eqref{V0} and \eqref{ipog}. Then the functional $f_{\lambda}$ satisfies the concrete-$(BPS)_{c}$.
\end{proposition}
\begin{proof}
Let $(u_h)\subset W^{1,p}_{{\rm rad}}(\R^{N})$ be a bounded sequence which satisfies the properties in \eqref{bps}.
Then, in turn, there exists a subsequence,  still denoted by $(u_h)$, 
converging  weakly in $\wsobr$, strongly in $L^q(\R^N)$ for any $q\in (p,p^*)$ and almost everywhere to a function $u\in\wsobr$. 
Moreover, we can apply the result in \cite{murboc} to obtain that $Du_{h}$ converges to $Du$ almost everywhere.
More precisely, since the variational formulation is here restricted to radial functions, this property follows
by arguing as in \cite[proof of Theorem 6.4]{sqradII}. Then, it is possible to follow 
the same arguments used in \cite[Step 2 of Lemma 2]{pelUMI} (see also \cite{squTO}) 
for bounded domains, in order to pass to the limit in the equation in \eqref{bps} and obtain in turn 
that $u$ satisfies the variational identity
\begin{align*}
\intr j_t(u,|Du|)\frac{Du}{|Du|}\cdot D\varphi &+\intr j_s(u,|Du|)\varphi  \\
 + \intr V(|x|)|u|^{p-2}uv &=\lambda\intr g(u)\varphi,
\quad\,\,
\forall \varphi\in C^{\infty}_{c, {\rm rad}}(\R^{N}).
\end{align*}
In fact, all the particular test functions built in \cite{pelUMI,squTO} to achieve this identity
are radial, since each $u_h$ is radial and $\varphi$ is a fixed radial function. 
Observe also that a function $\varphi\in W^{1,p}_{{\rm rad}}(\R^{N})\cap L^\infty(\R^N)$
can be approximated, in the $\|\cdot\|_{1,p}$ norm, by a sequence $(\varphi_m)\subset C^{\infty}_{c, {\rm rad}}(\R^{N})$
with  $\|\varphi_m\|_{L^\infty}\leq c(\varphi)$, for some positive constant $c(\varphi)$.
Whence, exploiting \eqref{jder}-\eqref{V0} and 
\eqref{ipog}, recalling that $u$ is radial and arguing as in \cite[Proposition 1]{pelUMI}, 
it follows that $u$ is an admissible test function, namely
\begin{equation}
	\label{identu}
\intr j_t(u,|Du|)|Du|+\intr j_s(u,|Du|)u+ \intr V(|x|)|u|^p=\lambda\intr g(u)u.
\end{equation}
Furthermore, taking into account that $u_{h}\in\wsobr$ and
exploiting conditions \eqref{ipog},  we can use 
\cite[Theorem A.I]{berlions} to obtain that 
$$
\lim_{h\to\infty}\intr g(u_h)u_h=\intr g(u)u.
$$
Observe that, applying by Fatou's lemma in view of \eqref{segno}-\eqref{V0} and \eqref{jxicoerc}, 
formula \eqref{identu} implies
\begin{align*}
\intr j_t(u,|Du|) |Du|+V(|x|)|u|^p
&\leq
\liminf_{h\to\infty}\Big\{\intr j_t(u_h,|Du_h|)|Du_h|
+V(|x|)|u_h|^p\Big\}
\\
&\leq
\limsup_{h\to\infty}\Big\{\intr j_t(u_h,|Du_h|)|Du_h|
+V(|x|)|u_h|^p\Big\}
\\
&\leq -\liminf_{h\to\infty}
\intr j_s(u_h,|Du_h|)u_h+\lim_{h\to\infty}\lambda\intr g(u_h)u_h
\\
&=-\intr j_s(u,|Du|)u+\lambda\intr g(u)u
\\
&=
\intr j_t(u,|Du|)|Du|+V(|x|)|u|^p.
\end{align*}
Then, taking into account \eqref{V0} and \eqref{jxicoerc}, it results
$$
\lim_{h\to\infty} \intr |Du_h|^p+m|u_h|^p =\intr |Du|^p+m|u|^p,
$$
giving the desired convergence of $(u_{h})$ to $u$
via the uniform convexity of $W^{1,p}(\R^N)$.
\end{proof}

\noindent
Next, we state the main technical tool for the proof of the first theorem.

\begin{lemma}\label{montrick}
Assume that conditions \eqref{j0}-\eqref{V0} and \eqref{ipog}-\eqref{Gpos} hold
and that $f_\lambda$ satisfies the concrete-$(BPS)_{c}$ for all $c\in\R$ and all $\lambda\in [\delta_0,1]$. Then there exists
a sequence $(\lambda_j,u_j)\subset [\delta_0,1]\times \wsobr$ 
with $\lambda_j\nearrow 1$ and where $u_j$ is
a distributional solution to 
\begin{equation}\label{eqradla}
-{\rm div}\Big[j_{t}(u,|Du|)\frac{Du}{|Du|}\Big]+j_{s}(u,|Du|)+V(|x|) |u|^{p-2}u=\lambda_j g(u)\quad\,\,\, \text{in $\R^{N}$},
\end{equation}
such that $f_{\lambda_j}(u_j)=c_{\lambda_j}$.
\end{lemma}
\begin{proof}
The result follows by applying \cite[Corollary 3.3]{montrick} to the minimax class defined in \eqref{minmaxclass}, with the choice of spaces
$X=S=V=\wsobr$ and by defining $u^H:=u$ and $u^{*}:=u$ as the identity maps. In fact, assumptions $({\mathcal H}_1)$ and $({\mathcal H}_2)$
are fulfilled thanks to Lemmas \ref{geo1} and \ref{geo2}.
Condition $({\mathcal H}_3)$ is implied by the structure of $f_{\lambda}$ as it can be verified by a straightforward 
direct computation. Finally assumption $({\mathcal H}_4)$ is evidently satisfied since
$u^H$ is the identity map. 
Since $X=\wsobr$, it turns out that, a priori, the solutions $(u_j)$ provided by \cite[Corollary 3.3]{montrick} are distributional with respect to test functions in 
$C^{\infty}_{c, {\rm rad}}(\R^{N})$. The fact that $u_j$ is, actually, a distributional solution with respect to any test function in 
$C^{\infty}_{c}(\R^{N})$ follows by \cite[Theorem 4.1 and end of the proof of Theorem 6.4]{sqradII}.
\end{proof}

\begin{proposition}
\label{continleft}
Assume \eqref{j0}, \eqref{V0} and \eqref{ipog}-\eqref{Gpos}.
The map $\lambda\to c_\lambda$ is non-increasing and continuous from the left.
\end{proposition}
\begin{proof}
	The fact that $c_\lambda$ is non-increasing trivially follows from the fact that $G\geq 0$.
	The proof of the left-continuity follows arguing by contradiction exactly as done in \cite[Lemma 2.3]{Jj}.
\end{proof}

\subsection{Proof of Theorem \ref{rad} concluded.}
Proposition \ref{comp} allows us to apply Lemma \ref{montrick} and 
obtain, in turn, a sequence $u_{j}$ of distributional solution of \eqref{eqradla} at the energy level 
$c_{\lambda_{j}}$. Following the argument in \cite[Lemma 4.1]{giasqu}
and applying \cite[Theorem 1 and Remark p.261]{serrin} one obtains
$u_{j}\in L^{\infty}_{{\rm loc}}(\R^{N})$ and then, via standard regularity arguments (see \cite{ladura}) $u_{j}\in C^{1,\alpha}(\R^N)$.
As a consequence, we can apply the Poh\v ozaev variational identity for $C^1$ solutions of equation~\eqref{eqradla}
stated in \cite[Lemma 1]{dms}, by choosing therein $h(x)=h_k(x)=H(x/k)x\in C^1_c(\R^N;\R^N)$, where 
$H\in C^1_c(\R^N)$ is such that $H(x)=1$ on $|x|\leq 1$ and $H(x)=0$ for $|x|\geq 2$. Letting $k\to\infty$
and taking into account conditions \eqref{j0}, \eqref{jder} and that $V'(|x|)|x|\in L^{N/p}(\R^N)$, we reach
\begin{align*}
& \intr j_{t}(u_{j},|Du_{j}|)|Du_{j}|-N\intr j(u_{j},|Du_{j}|)-\frac{N}{p}\intr V(|x|)|u_j|^p	\\
& +N\lambda_j\intr G(u_j)-\frac{1}{p}\intr V'(|x|)|x| |u_{j}|^{p}=0,\quad\text{for all $j\geq 1$}.
\end{align*}
In turn, each $u_{j}$ satisfies the following identity
\begin{equation*}
f_{\lambda}(u_{j})=
\frac1N\intr j_{t}(u_{j},|Du_{j}|)|Du_{j}|-\frac1{Np}\intr
 V'(|x|)|x| |u_{j}|^{p},\quad\text{for all $j\geq 1$}.
\end{equation*}
Since  $f_{\lambda}(u_{j})=c_{\lambda_{j}}$ and recalling \eqref{jxicoerc} one has
$$
\|Du_{j}\|_{L^{p}(\R^N)}^{p}\big(\alpha p {\mathcal S}
-\|V'(|x|)|x|\|_{L^{N/p}(\R^N)}\big)\leq pN {\mathcal S}c_{\lambda_j},
\,\,\quad\text{for all $j\geq 1$},
$$
where ${\mathcal S}$ is the best constant for the Sobolev embedding.
The last inequality, jointly with \eqref{V1} and Proposition \ref{continleft}, yields the 
existence of $A>0$ such that 
\beq\label{stima1}
\|Du_{j}\|_{L^{p}(\R^N)}\leq A,\quad\text{for all $j\geq 1$}.
\eeq 
Also, since $u_{j}$ solves \eqref{eqradla}, by testing it
with $u_j$ itself (which is admissible), \eqref{segno} and \eqref{jxicoerc} give
$$
\intr V(|x|)|u_{j}|^{p}-\lambda_{j}\intr g(u_{j})u_{j}\leq 0.
$$
So that, conditions \eqref{V0}, \eqref{crescita} and \eqref{stima1} yield, for any fixed $\eps<m$,
\beq\label{stima2}
(m-\lambda_{j}\eps)\|u_{j}\|_{L^{p}(\R^N)}^{p}\leq \lambda_{j}\frac{C_{\eps}}{\mathcal S^{p^{*}/p}}A^{p^{*}}.
\eeq
Since $(\lambda_{j})$ is bounded, by combining \eqref{stima1} and \eqref{stima2} we get that 
$(u_j)$ is bounded in $\wsobr$. In turn, let us observe that $(u_j)$ is a concrete-$(BPS)_{c_1}$ for the
functional $f_1$. In fact notice that, taking into account
that $G(u_j)$ remains bounded in $L^1(\R^N)$ due to inequality \eqref{crescita}, that $f_{\lambda_j}(u_j)=c_{\lambda_j}$ and recalling
Proposition~\ref{continleft}, it follows as $j\to\infty$
\begin{equation}
	\label{energia}
f_1(u_j)=f_{\lambda_j}(u_j)+(\lambda_j-1)\intr G(u_j)=c_{\lambda_j}+(\lambda_j-1)\intr G(u_j)=c_1+o(1).
\end{equation}
Furthermore, by defining $\hat w_j=(\lambda_j-1) g(u_j)\in W^{-1,p'}(\R^{N})$,
for every $v\in C^{\infty}_c(\R^{N})$ we have 
\begin{align}
	\label{equazione}
\langle f'_1(u_j),v\rangle & =\intr j_{t}(u_j,|Du_j|)\frac{Du_j}{|Du_j|}\cdot Dv
+\intr j_{s}(u_j,|Du_j|)v \\
& +\intr V(|x|) |u_j|^{p-2}u_jv-\intr g(u_j)v=\langle f'_{\lambda_j}(u_j),v\rangle+ \langle \hat w_j,v\rangle=\langle \hat w_j,v\rangle.     \notag
\end{align}
Then, since in light of \eqref{crescita} and \eqref{stima1}-\eqref{stima2}, $\hat w_j\to 0$ in $W^{-1,p'}(\R^{N})$ as $j\to\infty$, 
Proposition \ref{comp} applied to $f_1$ and with $c=c_1$ implies that there exists a function $u\in \wsobr$ such that, up to a subsequence, 
$(u_j)$ converges to $u$ strongly in $\wsobr$. On account of formulas \eqref{energia}-\eqref{equazione} and the continuity of $f_1$,
and by an application of Lebesgue's Theorem we conclude that $u$ is a nontrivial radial Mountain Pass
solution of \eqref{equarad}. Finally, $u$ is automatically nonnegative, as follows by testing \eqref{equarad}
with the admissible (by \cite[Proposition 3.1]{squTO} holding also for unbounded domains) test function $-u^-$, 
in view of \eqref{segno}, \eqref{jxicoerc} and the fact that $g(s)=0$ for every $s\leq 0$.

\section{Proof of Theorem \ref{sym}}
\noindent
Equation~\eqref{equarad} is investigated by studying the continuous functional $f_\lambda: W^{1,p}(\R^{N})\to \R$ 
with $f_\lambda(u)$ again defined as in \eqref{frad}
which, for $\lambda=1$, corresponds to the action functional associated to~\eqref{equarad}. 

\begin{definition}
Let $\lambda\in [\delta_{0},1]$, for some $\delta_0>0$, and $c\in \R$. We say that $f_{\lambda}$ satisfies the concrete-$(SBPS)_{c}$ condition 
if every bounded sequence $(u_{h})$ in $W^{1,p}(\R^{N})$ such that there exists
$w_{h}\in W^{-1,p'}(\R^{N})$ with $w_{h}\to 0$ as $h\to\infty$,
$$
f_{\lambda}(u_{h})\to c,\quad \langle f'_{\lambda}(u_{h}),v \rangle 
=\langle w_{h},v\rangle\quad\forall v\in C^{\infty}_{c}(\R^{N}), 
$$
and
\begin{equation}
	\label{simmcondtiondef}
\|u_h-u_h^*\|_{L^p(\R^N)\cap L^{p^*}(\R^N)}\to 0, 
\end{equation}
admits a strongly convergent subsequence. Here $u^*:=|u|^*$, where $*$ denoted the Schwarz symmetrization.
\end{definition}

\begin{proposition}\label{comp2}
Let $\lambda\in [\delta_{0},1]$, for some $\delta_0>0$, $c\in \R$ and assume that 
\eqref{j0}-\eqref{V0} and \eqref{ipog2} hold. Then the functional $f_{\lambda}$ satisfies the concrete-$(SBPS)_{c}$.
\end{proposition}
\begin{proof}
	Given a concrete-$(SBPS)_{c}$ sequence $(u_h)\subset W^{1,p}(\R^N)$,
	as in the proof of Proposition~\ref{comp}, up to a subsequence, $(u_h)$ converges
	to a $u$ weakly, almost everywhere and, in addition, $Du_h$ converges to $Du$
	almost everywhere. The main difference with respect to Proposition~\ref{comp}
	is that the crucial limit
\begin{equation}
	\label{implimit}
	\lim_h \int_{\R^N} g(u_h)u_h=\int_{\R^N} g(u)u,
\end{equation}
admits now a different justification. Since $(u_h^*)\subset W^{1,p}_{{\rm rad}}(\R^{N})$ and $(u_h)$ is bounded
in $W^{1,p}(\R^N)$, then $(u_h^*)$ is bounded in $W^{1,p}(\R^N)$ too
by virtue of Polya-Szeg\"o inequality. Therefore, since for every $p<q<p^*$ the injection map $i:W^{1,p}_{{\rm rad}}(\R^N)\to L^q(\R^N)$ 
is completely continuous, up to a subsequence, it follows that 
$u_h^*\to z$ in $L^q(\R^N)$ as $h\to\infty$ for some $z\in L^q(\R^N)$, for $p<q<p^*$.
Due to $\|u_h-u_h^*\|_{L^p\cap L^{p^*}(\R^N)}\to 0$ we get
$u_h\to z$ in $L^q(\R^N)$, as
\begin{equation*}
\|u_h-z\|_{L^q(\R^N)} \leq C\|u_h-u_h^*\|_{L^p\cap L^{p^*}(\R^N)}+\|u_h^*-z\|_{L^q(\R^N)}.
\end{equation*}
Of course $z=u$, allowing to conclude that
\begin{equation}
	\label{startlimi1}
u_h\to u\quad\text{in $L^q(\R^N)$ as $h\to\infty$, \,\,\, for every $p<q<p^*$.}
\end{equation}
In light of \eqref{startlimi1}, for a $p<q<p^*$ there exists $\zeta \in L^q(\R^N)$, $\zeta\geq 0$, such that
$|u_h|\leq\zeta$ for every $h\geq 1$. In turn, by assumption \eqref{ipog2}, for all $\eps>0$ 
there exists $C_\eps\in\R$ with
\begin{equation*}
 \eps|u_h|^p+C_\eps \zeta^q - g(u_h)u_h\geq 0.
\end{equation*}
Then, by Fatou's Lemma, by the arbitrariness of $\eps$ and the boundedness of $(u_h)$ in $L^p(\R^N)$, 
$$
\limsup_h \intr g(u_h)u_h\leq \intr g(u)u.
$$
Of course, since $g(u_h)u_h\geq 0$, again by Fatou's Lemma one also has
$$
\liminf_h \intr g(u_h)u_h\geq \intr g(u)u,
$$
concluding the proof of formula \eqref{implimit}
\end{proof}

\noindent
Next, we state the main technical tool for the proof of the second theorem.
\begin{lemma}\label{montrick1}
Assume that conditions \eqref{j0}-\eqref{V0} and \eqref{ipog2}-\eqref{pol} hold
and that $f_\lambda$ satisfies the concrete-$(SBPS)_{c}$ for all $c\in\R$ and all $\lambda\in [\delta_0,1]$. Then there exists
a sequence $(\lambda_j,u_j)\subset [\delta_0,1]\times W^{1,p}(\R^N)$ with $\lambda_j\nearrow 1$ where $u_j$ is
a distributional solution of
\begin{equation*}
-{\rm div}\Big[j_{t}(u,|Du|)\frac{Du}{|Du|}\Big]+j_{s}(u,|Du|)+V(|x|) u^{p-1}=\lambda_j g(u)\quad\,\,\, \text{in $\R^{N}$},
\end{equation*}
such that $f_{\lambda_j}(u_j)=c_{\lambda_j}$ and $u_j=u_j^*$.
\end{lemma}
\begin{proof}
The result follows by applying \cite[Corollary 3.3]{montrick} with the following choice of spaces:
$X=W^{1,p}(\R^N)$, $S=W^{1,p}(\R^N,\R^+)$ and $V=L^p\cap L^{p^*}(\R^N)$. In fact, it is readily
verified that assumptions $({\mathcal H}_1)$-$({\mathcal H}_4)$ in \cite[section 3.1]{montrick} are fulfilled
with $u^H=|u|^H$, where $v^H$ denotes the standard polarization of $v\geq 0$ and with $u^*=|u|^*$ where $v^*$ denotes 
the Schwarz symmetrization of $v\geq 0$.
Condition $({\mathcal H}_1)$ is just the continuity of the functionals $f_\lambda$. Condition $({\mathcal H}_2)$ is satisfied 
since Lemma~\ref{geo1} and Lemma~\ref{geo2} hold with the same proof (notice that the function
$z$ in the proof of Lemma~\ref{geo1} satisfies $z=z^*$). Condition $({\mathcal H}_3)$ follows, as in the proof of Lemma~\ref{montrick}
by a simple direct computation.
Assumption $({\mathcal H}_4)$ is satisfied by \eqref{pol} and standard arguments (see also \cite[Remark 3.4]{montrick}).
Notice that the function $w=\gamma(1)=z(x/t_0)$ detected in Lemma~\ref{geo1} and used to build the minimax class $\Gamma$
is radially symmetric and radially decreasing, so that $w^H=w$ for
every half space $H$, as required in $({\mathcal H}_4)$.
\end{proof}

\subsection{Proof of Theorem \ref{sym} concluded}
The proof goes along the lines of the proof of Theorem~\ref{rad} by simple adaptations of the
preparatory results contained in Section \ref{proofradthm} to the new setting. With respect to 
the main differences in the proofs, it is sufficient
to replace Proposition~\ref{comp} with Proposition~\ref{comp2} and Lemma~\ref{montrick} with Lemma~\ref{montrick1}.

\begin{remark}\rm
	\label{finalremark}
In the notations $c$ and $c_{{\rm rad}}$ mentioned at the end of the introduction,
we always have $c\leq  c_{{\rm rad}}$. On the other hand,
when $V$ is constant and the function $t\mapsto j(s,t)$ is $p$-homogeneous, then $c\geq  c_{{\rm rad}}$.
In fact, let $u_r$ be a radial solution at level $c$ provided by Theorem \ref{sym}, namely $f_1(u_r)=c$. Then, defining the radial curve
$\gamma_r(t)(x):=u_r(x/tt_0)$, which belongs to $C([0,1],W^{1,p}_{{\rm rad}}(\R^N))$ for a suitable $t_0>1$ 
and arguing as in \cite[Step I, proof of Theorem 3.2]{giasqu} through Poh\v ozaev identity, it follows that 
$$
c=f_1(u_r)=\max_{t\in [0,1]} f_1(\gamma_r(t)),
$$
immediately yielding $c\geq  c_{{\rm rad}}$, as desired.
\end{remark}

\vskip6pt
\noindent
{\bf Acknowledgment.} The authors wish to thank Jean Van Schaftingen for a useful discussion about the comparison
between the Mountain Pass levels $c_{{\rm rad}}$ and $c$.

\end{document}